\documentclass[12pt,reqno]{amsart}
\usepackage{enumerate, latexsym, amsmath, amsfonts, amssymb, amsthm, color}
\def\pmod #1{\ ({\rm{mod}}\ #1)}
\def\Z{\Bbb Z}

\def\l{\left}
\def\r{\right}
\def\bg{\bigg}
\def\({\bg(}
\def\){\bg)}

\def\f{\frac}

\def\bi{\binom}

\def\eq{\equiv}

\def\jacob #1#2{\left(\frac{#1}{#2}\right)}

\theoremstyle{plain}
\newtheorem{theorem}{Theorem}

\newtheorem{lemma}{Lemma}

\theoremstyle{definition}
\newtheorem*{acknowledgment}{Acknowledgments}
\theoremstyle{remark}

\vspace{4mm}

\begin{document}

\title
{Supercongruences on some binomial sums involving Lucas sequences}

\author{Guo-Shuai Mao}
\address {Department of Mathematics, Nanjing
University, Nanjing 210093, People's Republic of China}
\email{mg1421007@smail.nju.edu.cn}
\author{Hao Pan}
\address{Department of Mathematics, Nanjing
University, Nanjing 210093, People's Republic of China}
\email{haopan79@zoho.com}


\begin{abstract}
In this paper, we confirm several conjectured congruences of Sun concerning the divisibility of binomial sums. For example, with help of a quadratic hypergeometric transformation, we prove that
$$
\sum_{k=0}^{p-1}\binom{p-1}k\binom{2k}k^2\frac{P_k}{8^k}\equiv0\pmod{p^2}
$$
for any prime $p\equiv 7\pmod{8}$, where $P_k$ is the $k$-th Pell number. Further, we also propose three new congruences of the same type.
\end{abstract}

\maketitle
\section{Introduction}
\setcounter{lemma}{0}
\setcounter{theorem}{0}
\setcounter{corollary}{0}
\setcounter{remark}{0}
\setcounter{equation}{0}
\setcounter{conjecture}{0}

In \cite{Mortenson03a}, with help of the Gross-Koblitz formula, Mortenson solved a conjecture of Rodriguez-Villegas \cite{RV03} as follows:
$$
\sum_{k=0}^{p-1}\binom{2k}{k}^2\frac1{16^k}\equiv\jacob{-1}{p}\pmod{p^2}
$$
for every odd prime $p$,
where $\jacob{\cdot}{p}$ denotes the Legendre symbol. Subsequently, the similar congruences were widely studied.
For the progress of this topic, the reader may refer to \cite{Mortenson03a, Mortenson05, OS09, Long11, Sh11, Sw11a, Sw11b, McCarthy12, Sh13, MT13, GZ14, Sh14}.
In \cite{Sw09}, Sun proposed many conjectured congruences on the sums of binomial coefficients.
Some of those conjectures are of the form
$$
\sum_{k=0}^{p-1}\binom{p-1}{k}\binom{2k}{k}^2a_n\equiv 0\pmod{p^2}.
$$
For example, Sun conjectured that
$$
\sum_{k=0}^{p-1}\binom{p-1}k\binom{2k}k^2\frac{\chi_3(k)}{16^k}\equiv0\pmod{p^2}
$$
for any prime $p\equiv 1\pmod{12}$, where $\chi_3(k)$ equals to the Legendre symbol $\jacob{k}{3}$.

The main purpose of this paper is to confirm the following conjectures of Sun.
\begin{theorem}\label{thm1}  Suppose that $p$ is a prime.\medskip

\noindent{\rm (i)} If $p\equiv3\pmod 4$, then
\begin{equation}\label{cong2}\sum_{k=0}^{p-1}\binom{p-1}k\binom{2k}k^2\frac{1}{(-8)^k}\equiv0\pmod{p^2}.
\end{equation}

\noindent{\rm (ii)} If $p\eq1\pmod{12}$, then
\begin{equation}\label{cong3}\sum_{k=0}^{p-1}\binom{p-1}k\binom{2k}k^2\frac{\chi_3(k)}{16^k}\equiv0\pmod{p^2}.
\end{equation}

\noindent{\rm (iii)} If $p\eq7\pmod{8}$, then
\begin{equation}\label{congp}\sum_{k=0}^{p-1}\binom{p-1}k\binom{2k}k^2\frac{P_k}{8^k}\equiv0\pmod{p^2},
\end{equation}
where the Pell number $P_k$ is given by
$$
P_0=0,\ P_1=1,\ P_{n}=2P_{n-1}+P_{n-2}\text{ for }n\geq 2.
$$
\noindent(iv) If $p\eq11\pmod{12}$, then
\begin{equation}\label{congr}\sum_{k=0}^{p-1}\binom{p-1}k\frac{R_k}{(-4)^k}\binom{2k}k^2\equiv0\pmod{p^2},
\end{equation}
where $R_k$ is given by
$$
R_0=2,\ R_1=4,\ R_{n}=4R_{n-1}-R_{n-2}\text{ for }n\geq 2.
$$
\end{theorem}
We mention that (\ref{cong2}), (\ref{cong3}), (\ref{congp}) and (\ref{congr}) respectively belong to Conjecture 5.5 of \cite{Sw11a} and Conjectures A56, A57, A63 of \cite{Sw09}.

The sequences $\{P_n\}$ and $\{R_n\}$ in Theorem \ref{thm1} both belong to the second-order linear recurrence sequence.
In general, define the Lucas sequences $\{U_n(a,b)\}$ and $\{V_n(a,b)\}$ by
$$
U_0(a,b)=0,\ U_1(a,b)=1,\ U_n(a,b)=aU_{n-1}(a,b)-bU_{n-2}(a,b)\text{ for }n\geq 2,
$$
and
$$
V_0(a,b)=2,\ V_1(a,b)=a,\ V_n(a,b)=aV_{n-1}(a,b)-bV_{n-2}(a,b)\text{ for }n\geq 2.
$$
Clearly $P_n=U_n(2,-1)$ and $R_n=V_n(4,1)$.
In fact, it is also easy to see that $-(-2)^{n+1}=V_n(-4,4)$ and $\chi_3(n)=U_n(-1,1)$. So it is natural to study
the arithmetical properties of
$$
\sum_{k=0}^{\frac{p-1}{2}}\binom{p-1}{k}\binom{2k}{k}^2\frac{U_k(a,b)}{16^k}\text{\qquad and\qquad}
\sum_{k=0}^{\frac{p-1}{2}}\binom{p-1}{k}\binom{2k}{k}^2\frac{V_k(a,b)}{16^k}.
$$

Define the $n$-th harmonic number
$$
H_n=\sum_{k=1}^n\frac1k.
$$
In particular, set $H_0=0$.
We have
\begin{theorem}\label{thm2} Suppose that $p$ is an odd prime and $a,b\in\Z$. Then
\begin{align}\label{congvn}
&\sum_{k=0}^{\frac{p-1}{2}}\binom{p-1}{k}\binom{2k}{k}^2\frac{V_k(a,b)}{16^k}\notag\\
\equiv&(-1)^{\frac{p-1}{2}}16^{p-1}\sum_{j=0}^{\frac{p-1}{2}}\binom{\frac{p-1}{2}}{j}^2V_j(a+2,a+b+1)(1+2pH_{2j})
\pmod{p^2}.
\end{align}
Furthermore, if $p^2$ doesn't divide $a^2-4b$, then
\begin{align}\label{congun}
&\sum_{k=0}^{\frac{p-1}{2}}\binom{p-1}{k}\binom{2k}{k}^2\frac{U_k(a,b)}{16^k}\notag\\
\equiv&(-1)^{\frac{p-1}{2}}16^{p-1}\sum_{j=0}^{\frac{p-1}{2}}\binom{\frac{p-1}{2}}{j}^2U_j(a+2,a+b+1)(1+2pH_{2j})
\pmod{p^2}.
\end{align}
\end{theorem}

With the help of Theorem \ref{thm2}, here we can obtain three new congruences of the same type.
\begin{theorem}\label{thm1n} Suppose that $p$ is a prime.\medskip

{\rm (i)} If $p\equiv7\pmod{8}$, then
\begin{equation}\label{congw}
\sum_{k=0}^{p-1}\binom{p-1}k\binom{2k}k^2\frac{W_k}{4^k}\equiv0\pmod{p^2},
\end{equation}
where $W_k$ is given by
$$
W_0=0,\ W_1=1,\ W_{n}=8W_{n-1}+2W_{n-2}\text{ for }n\geq 2.
$$
{\rm (ii)} If $p\equiv1\pmod{6}$, then
\begin{equation}\label{congm}
\sum_{k=0}^{p-1}\binom{p-1}k\binom{2k}k^2\frac{(-1)^kM_k}{16^k}\equiv0\pmod{p^2},
\end{equation}
where $M_k$ is given by
$$
M_0=0,\ M_1=1,\ M_{n}=3M_{n-1}-3M_{n-2}\text{ for }n\geq 2.
$$
{\rm (iii)} If $p\equiv7\pmod{12}$, then
\begin{equation}\label{congd3}
\sum_{\substack{0\leq k\leq p-1\\
k\equiv 0\pmod{3}}}\binom{p-1}k\binom{2k}k^2\frac{1}{16^k}\equiv
\frac13\sum_{k=0}^{p-1}\binom{p-1}k\binom{2k}k^2\frac{1}{16^k}\pmod{p^2}.
\end{equation}
\end{theorem}

First, the proof of Theorem \ref{thm2} will be given in Section 2.
It is not difficult to check that $2^nP_n=U_n(4,-4)$ and $P_{2n}=U_n(6,1)$.
Then according to Theorem \ref{thm2}, in order to prove (\ref{congp}),
we only need to show that
$$
\sum_{j=0}^{\frac{p-1}{2}}\binom{\frac{p-1}{2}}{j}^2P_{2j}\equiv0\pmod{p^2}
\text{\ \ and\ \ }
\sum_{j=0}^{\frac{p-1}{2}}\binom{\frac{p-1}{2}}{j}^2P_{2j}H_{2j}\equiv0\pmod{p}.
$$
However, as we shall see later, the former one is not easy to prove. So in the third section,
we shall firstly establish an auxiliary lemma, by using some quadratic hypergeometric transformations.
Further, the similar divisible congruences for $R_n$ and $W_n$ will be also proved.
Finally, in Section 4, we shall conclude the proofs of Theorems \ref{thm1} and \ref{thm1n}.

\section{Proof of Theorem \ref{thm2}}
\setcounter{lemma}{0}
\setcounter{theorem}{0}
\setcounter{corollary}{0}
\setcounter{remark}{0}
\setcounter{equation}{0}
\setcounter{conjecture}{0}
Below we always assume that $p$ is an odd prime. In this section, we shall prove
\begin{theorem}\label{thm3}
\begin{align}\label{2kkzh2j}
&\sum_{k=0}^{\frac{p-1}{2}}\binom{p-1}{k}\binom{2k}{k}^2\frac{(z-1)^k}{16^k}\notag\\
\equiv&(-1)^{\frac{p-1}{2}}16^{p-1}\sum_{j=0}^{\frac{p-1}{2}}
z^j\binom{\frac{p-1}{2}}{j}^2(1+2pH_{2j})\pmod{p^2}.
\end{align}
\end{theorem}
\begin{lemma}\label{2kkz}
\begin{equation}\label{2kkze}
\sum_{k=0}^{\frac{p-1}{2}}\binom{2k}{k}^2\frac{(1-z)^k}{16^k}\equiv
(-1)^{\frac{p-1}{2}}4^{p-1}\sum_{j=0}^{\frac{p-1}{2}}z^{j}
\binom{\frac{p-1}{2}}{j}^2(1+pH_{\frac{p-1}{2}-j})\pmod{p^2}.
\end{equation}
\end{lemma}
\begin{proof}
Clearly
\begin{align*}
\binom{2k}{k}^2=16^k\binom{-\frac12}{k}^2
\equiv&
\frac{16^k}{(k!)^2}\prod_{j=0}^{k-1}\bigg(\bigg(-\frac12-j\bigg)^2-\frac{p^2}{4}\bigg)\\
=&\frac{16^k}{(k!)^2}\prod_{j=0}^{k-1}\bigg(\frac{p-1}2-j\bigg)\bigg(-\frac{p+1}2-j\bigg)\\
=&
16^k\binom{\frac{p-1}{2}}{k}\binom{-\frac{p+1}{2}}{k}\pmod{p^2}.
\end{align*}
It follows that
\begin{align*}
\sum_{k=0}^{\frac{p-1}{2}}\binom{2k}{k}^2\frac{(1-z)^k}{16^k}
\equiv&
\sum_{k=0}^{\frac{p-1}{2}}\binom{\frac{p-1}{2}}{k}\binom{-\frac{p+1}{2}}{k}\sum_{j=0}^k\binom{k}{j}(-z)^j\\
=&\sum_{j=0}^{\frac{p-1}{2}}\binom{\frac{p-1}{2}}{j}(-z)^j\sum_{k=j}^{\frac{p-1}{2}}\binom{\frac{p-1}{2}-j}{k-j}\binom{-\frac{p+1}{2}}{k}\pmod{p^2}.
\end{align*}
In view of the Chu-Vandemonde identity \cite[(5.27)]{GKP94}, we have
\begin{align*}
\sum_{k=j}^{\frac{p-1}{2}}\binom{\frac{p-1}{2}-j}{\frac{p-1}{2}-k}\binom{-\frac{p+1}{2}}{k}
=\binom{-1-j}{\frac{p-1}{2}}=(-1)^{\frac{p-1}{2}}\binom{\frac{p-1}{2}+j}{\frac{p-1}{2}}.
\end{align*}
Since
$$
\binom{\frac{p-1}{2}+j}{\frac{p-1}{2}}=
\frac{\binom{p-1}{\frac{p-1}{2}}\binom{\frac{p-1}{2}}{j}}{\binom{p-1}{\frac{p-1}{2}+j}}
$$
and
$$
\binom{p-1}{k}=\prod_{i=1}^k\bigg(\frac{p}{i}-1\bigg)\equiv (-1)^k(1-pH_k)\pmod{p^2},
$$
we obtain that
\begin{align*}
\sum_{k=0}^{\frac{p-1}{2}}\binom{2k}{k}^2\cdot\frac{(1-z)^k}{16^k}\equiv&
\sum_{j=0}^{\frac{p-1}{2}}\binom{\frac{p-1}{2}}{j}(-z)^j\cdot
(-1)^{\frac{p-1}{2}}\frac{\binom{p-1}{\frac{p-1}{2}}\binom{\frac{p-1}{2}}{j}}{\binom{p-1}{\frac{p-1}{2}-j}}\\
\equiv&
\binom{p-1}{\frac{p-1}{2}}\sum_{j=0}^{\frac{p-1}{2}}z^{j}
\binom{\frac{p-1}{2}}{j}^2(1+pH_{\frac{p-1}{2}-j})\pmod{p^2}.
\end{align*}
Using the classical Morley congruence \cite{Morley95}
$$
\binom{p-1}{\frac{p-1}{2}}\equiv (-1)^{\frac{p-1}{2}}4^{p-1}\pmod{p^2},
$$
we get (\ref{2kkze}).
\end{proof}
\begin{lemma}\label{2kkzh}
\begin{align}\label{2kkzhe}
&(-1)^{\frac{p+1}2}\sum_{k=0}^{\frac{p-1}{2}}(1-z)^k\binom{2k}{k}^2\frac{H_k}{16^k}\notag\\
\equiv&\frac{2^{p+1}-4}{p}
\sum_{j=0}^{\frac{p-1}{2}}z^j\binom{\frac{p-1}{2}}{j}^2+
\sum_{j=0}^{\frac{p-1}{2}}z^j\binom{\frac{p-1}{2}}{j}^2H_j\pmod{p}.
\end{align}
\end{lemma}
\begin{proof}
Clearly
$$
\binom{2k}{k}=(-4)^k\binom{-\frac12}{k}\equiv(-4)^k\binom{\frac{p-1}2}{k}\pmod{p}.
$$
Hence
\begin{align*}
\sum_{k=0}^{\frac{p-1}{2}}\binom{2k}{k}^2\frac{(1-z)^kH_k}{16^k}\equiv&\sum_{k=0}^{\frac{p-1}{2}}\binom{\frac{p-1}{2}}{k}^2H_k\sum_{j=0}^k\binom{k}j(-z)^j\\
\equiv&\sum_{j=0}^{\frac{p-1}{2}}\binom{\frac{p-1}{2}}{j}(-z)^j
\sum_{k=j}^{\frac{p-1}{2}}\binom{\frac{p-1}{2}-j}{k-j}\binom{\frac{p-1}{2}}{k}H_k\pmod{p}.
\end{align*}
Apparently
$$
\lim_{t\to0}\frac{d}{dt}\bigg(\binom{\frac{p-1}{2}-t}{\frac{p-1}{2}-k}\bigg)=
-\binom{\frac{p-1}{2}}{\frac{p-1}{2}-k}\sum_{i=0}^{\frac{p-1}{2}-k-1}\frac{1}{\frac{p-1}{2}-i}=
\binom{\frac{p-1}{2}}{k}(H_{k}-H_{\frac{p-1}{2}}).
$$
Note that
$$
\sum_{k=j}^{\frac{p-1}{2}}\binom{\frac{p-1}{2}-j}{k-j}
\binom{\frac{p-1}{2}-t}{\frac{p-1}{2}-k}=
\binom{p-1-j-t}{\frac{p-1}{2}-j},
$$
and
\begin{align*}
\lim_{t\to0}\frac{d}{dt}\bigg(\binom{p-1-j-t}{\frac{p-1}{2}-j}\bigg)=&
-\binom{p-1-j}{\frac{p-1}{2}-j}\sum_{i=0}^{\frac{p-1}{2}-j-1}\frac1{p-1-j-i}\\
\equiv&
\binom{p-1-j}{\frac{p-1}{2}-j}\sum_{i=0}^{\frac{p-1}{2}-j-1}\frac1{i+j+1}\\
=&
\binom{p-1-j}{\frac{p-1}{2}-j}(H_{\frac{p-1}{2}}-H_j)\pmod{p}.
\end{align*}
We obtain that
\begin{align*}
&\sum_{k=j}^{\frac{p-1}{2}}\binom{\frac{p-1}{2}-j}{k-j}\binom{\frac{p-1}{2}}{k}H_k\\
=&\sum_{k=j}^{\frac{p-1}{2}}\binom{\frac{p-1}{2}-j}{\frac{p-1}{2}-k}\binom{\frac{p-1}{2}}{k}H_{\frac{p-1}2}
+\lim_{t\to0}\frac{d}{dt}\bigg(\sum_{k=j}^{\frac{p-1}{2}}\binom{\frac{p-1}{2}-j}{k-j}
\binom{\frac{p-1}{2}-t}{\frac{p-1}{2}-k}\bigg)
\\
\equiv&\binom{p-1-j}{\frac{p-1}{2}-j}
(2H_{\frac{p-1}{2}}-H_j)=(-1)^{\frac{p-1}{2}-j}
\binom{-\frac{p+1}2}{\frac{p-1}{2}-j}
(2H_{\frac{p-1}{2}}-H_j)\\
\equiv&(-1)^{\frac{p-1}{2}-j}
\binom{\frac{p-1}2}{\frac{p-1}{2}-j}
(2H_{\frac{p-1}{2}}-H_j)\pmod{p}.
\end{align*}
Thus (\ref{2kkzhe}) immediately follows from a well-known congruence of Lehmer \cite{Lehmer38}:
$$
H_{\frac{p-1}{2}}\equiv-\frac{2^p-2}{p}\pmod{p}.
$$
\end{proof}
\begin{proof}[Proof of Theorem \ref{thm3}] By the Fermat little theorem,
$$
4^{p-1}=(1+2^{p-1}-1)^2\equiv1+2(2^{p-1}-1)\pmod{p^2}.
$$
Combining Lemmas \ref{2kkz} and \ref{2kkzh}, we obtain that
\begin{align*}
&(-1)^{\frac{p-1}{2}}\sum_{k=0}^{\frac{p-1}{2}}\binom{p-1}{k}\binom{2k}{k}^2\frac{(z-1)^k}{16^k}\\
\equiv&
(-1)^{\frac{p-1}{2}}\sum_{k=0}^{\frac{p-1}{2}}\binom{2k}{k}^2\frac{(1-z)^k}{16^k}(1-pH_k)\\
\equiv&(3\cdot 2^{p}-5)\sum_{j=0}^{\frac{p-1}{2}}z^j\binom{\frac{p-1}{2}}{j}^2+
p\sum_{j=0}^{\frac{p-1}{2}}z^j\binom{\frac{p-1}{2}}{j}^2(H_j+H_{\frac{p-1}{2}-j})\pmod{p^2}.
\end{align*}
On the other hand,
\begin{align*}
H_{\frac{p-1}{2}-j}=&\sum_{i=j+1}^{\frac{p-1}{2}}\frac{1}{\frac{p+1}{2}-i}\equiv
-\sum_{i=j+1}^{\frac{p-1}{2}}\frac{2}{2i-1}\\
=&\sum_{i=1}^{j}\frac{2}{2i-1}-\sum_{i=1}^{\frac{p-1}{2}}\frac{2}{2i-1}=(2H_{2j}-H_j)-
(2H_{p-1}-H_{\frac{p-1}2})\\
\equiv&2H_{2j}-H_j+H_{\frac{p-1}2}\pmod{p},
\end{align*}
where in the last step we use the well-known fact
$$
H_{p-1}\equiv0\pmod{p}.
$$
Thus
\begin{align*}
&(-1)^{\frac{p-1}{2}}\sum_{k=0}^{\frac{p-1}{2}}\binom{p-1}{k}\binom{2k}{k}^2\frac{(z-1)^k}{16^k}\\
\equiv&
(3\cdot 2^{p}-5)\sum_{j=0}^{\frac{p-1}{2}}z^j\binom{\frac{p-1}{2}}{j}^2+
p\sum_{j=0}^{\frac{p-1}{2}}z^j\binom{\frac{p-1}{2}}{j}^2(2H_{2j}+H_{\frac{p-1}2})\\
\equiv&(2^{p+1}-3)\sum_{j=0}^{\frac{p-1}{2}}z^j\binom{\frac{p-1}{2}}{j}^2+
2p\sum_{j=0}^{\frac{p-1}{2}}z^j\binom{\frac{p-1}{2}}{j}^2H_{2j}
\pmod{p^2}.
\end{align*}
Finally, we have
$$
16^{p-1}=(1+2^{p-1}-1)^4\equiv1+4(2^{p-1}-1)\pmod{p^2}.
$$
The proof of (\ref{2kkzh2j}) is concluded.\end{proof}

Now Theorem \ref{thm2} is an easy consequence of Theorem \ref{thm3}. Let $\alpha$ and $\beta$ be the two roots of
the equation $x^2-ax+b=0$. Then $\alpha+1$ and $\beta+1$ are also the two roots of $x^2-(a+2)x+(b+a+1)=0$. It is well-known (cf. \cite{Ribenboim00}) that
$$
V_n(a,b)=\alpha^n+\beta^n,\qquad
V_n(a+2,b+a+1)=(\alpha+1)^n+(\beta+1)^n
$$
for each $n\geq 0$. By Theorem \ref{thm3},
\begin{align*}
&\sum_{k=0}^{\frac{p-1}{2}}\binom{p-1}{k}\binom{2k}{k}^2\frac{\alpha^k+\beta^k}{16^k}\\
\equiv&(-1)^{\frac{p-1}{2}}16^{p-1}\sum_{j=0}^{\frac{p-1}{2}}
((\alpha+1)^k+(\beta+1)^k)\binom{\frac{p-1}{2}}{j}^2(1+2pH_{2j})\pmod{p^2}.
\end{align*}
Then (\ref{congvn}) is derived.

Furthermore, we also have
$$
U_n(a,b)=\frac{\alpha^n-\beta^n}{\alpha-\beta},\qquad
U_n(a+2,b+a+1)=\frac{(\alpha+1)^n-(\beta+1)^n}{\alpha-\beta}.
$$
Suppose that $p^2$ doesn't divides $a^2-4b$. Then $\alpha-\beta=\pm\sqrt{a^2-4b}$ is not divisible by $p$.
So we have
\begin{align*}
&\sum_{k=0}^{\frac{p-1}{2}}\binom{p-1}{k}\binom{2k}{k}^2\frac{\alpha^k-\beta^k}{16^k(\alpha-\beta)}\\
\equiv&(-1)^{\frac{p-1}{2}}16^{p-1}\sum_{j=0}^{\frac{p-1}{2}}
\frac{(\alpha+1)^k-(\beta+1)^k}{\alpha-\beta}\binom{\frac{p-1}{2}}{j}^2(1+2pH_{2j})\pmod{p^{2}},
\end{align*}
by noting that both sides of the above congruences are factly rational $p$-integers.

\section{Quadratic hypergeometric transformations}
\setcounter{lemma}{0}
\setcounter{theorem}{0}
\setcounter{corollary}{0}
\setcounter{remark}{0}
\setcounter{equation}{0}
\setcounter{conjecture}{0}

In this section, we shall use the quadratic hypergeometric transformations to deduce some auxiliary results on $P_n$, $R_n$ and $W_n$, which is necessary for the proof of (\ref{congp}), (\ref{congr}) and (\ref{congw}). Define the hypergeometric function
\begin{equation}\label{F21def}
{}_2F_1\bigg(\begin{matrix}a&b\\ &c\end{matrix}\bigg|z\bigg)=\sum_{k=0}^\infty\frac{(a)_k(b)_k}{(c)_k}\cdot\frac{z^k}{k!},
\end{equation}
where $c\not\in\{0,-1,-2,\ldots\}$ and
$$
(a)_k=\begin{cases}a(a+1)\cdots(a+k-1),&\text{if }k\geq1,\\
1,&\text{if }k=0.
\end{cases}
$$
Clearly (\ref{F21def}) is convergent whenever $|z|<1$.
And if $n$ is a non-negative integer, then
$$
\sum_{k=0}^n\binom{n}{k}\binom{m}{k}z^k=\sum_{k=0}^\infty\frac{(-n)_k(-m)_k}{(1)_k}\cdot\frac{z^k}{k!}
={}_2F_1\bigg(\begin{matrix}-n&-m\\ &1\end{matrix}\bigg|z\bigg).
$$

\begin{lemma} {\rm (i)} Suppose that $p\equiv 5,7\pmod{8}$ is prime. Then
\begin{equation}\label{alphacong1}
\sum_{k=0}^{\frac{p-1}2}\binom{\frac{p-1}2}{k}^2(1+\sqrt{2})^{2k}\equiv
0\pmod{p}.
\end{equation}

{\rm (ii)} Suppose that $p\equiv2\pmod{3}$ is a prime.
Then
\begin{equation}\label{alphacong2}
\sum_{k=0}^{\frac{p-1}{2}}\binom{\frac{p-1}{2}}{k}^2(-1)^k(2+\sqrt{3})^{2k}\equiv0\pmod{p}.
\end{equation}

{\rm (iii)} Suppose that $p\equiv3\pmod{4}$ is a prime. Then
\begin{equation}\label{alphacong3}
\sum_{k=0}^{\frac{p-1}2}\binom{\frac{p-1}2}{k}^2(3+2\sqrt{2})^{2k}\equiv0\pmod{p}.
\end{equation}
\end{lemma}
\begin{proof}
(i) We need the following quadratic transformation for the hypergeometric functions \cite[(3.1.4)]{AAR99}:
\begin{equation}\label{qthf1}
{}_2F_1\bigg(\begin{matrix}a&b\\ &a-b+1\end{matrix}\bigg|z\bigg)=
(1-z)^{-a}{}_2F_1\bigg(\begin{matrix}\frac12a&\frac12a-b+\frac12\\ &a-b+1\end{matrix}\bigg|-\frac{4z}{(1-z)^2}\bigg),
\end{equation}
where $|z|<1$.
Let $z=(1+\sqrt{2})^2$. It is easy to check that
$$
-\frac{4z}{(1-z)^2}=-1.
$$

Suppose that $p\equiv5\pmod{8}$. Substituting $a=b=-\frac{p-1}{2}$ in (\ref{qthf1}), we get
\begin{align*}
\sum_{k=0}^{\frac{p-1}2}\binom{\frac{p-1}2}{k}^2z^k=&
{}_2F_1\bigg(\begin{matrix}-\frac{p-1}2&-\frac{p-1}2\\ &1\end{matrix}\bigg|z\bigg)\\
=&
(1-z)^{\frac{p-1}{2}}{}_2F_1\bigg(\begin{matrix}-\frac{p-1}4&\frac{p+1}4\\ &1\end{matrix}\bigg|-1\bigg).
\end{align*}
Notice that both ${}_2F_1\Big(\begin{matrix}-\frac{p-1}2&-\frac{p-1}2\\ &1\end{matrix}\Big|z\Big)$ and ${}_2F_1\Big(\begin{matrix}-\frac{p-1}4&\frac{p+1}4\\ &1\end{matrix}\Big|-1\Big)$ are factly the finite summations. So
the requirement $|z|<1$ can be ignored here. Then (\ref{alphacong1}) follows from that
\begin{align*}
{}_2F_1\bigg(\begin{matrix}-\frac{p-1}4&\frac{p+1}4\\ &1\end{matrix}\bigg|-1\bigg)=
&\sum_{k=0}^{\frac{p-1}4}\binom{\frac{p-1}4}{k}\binom{-\frac{p+1}4}{k}(-1)^k\\
\equiv&
\sum_{k=0}^{\frac{p-1}4}\binom{\frac{p-1}4}{k}^2(-1)^k=0\pmod{p},
\end{align*}
where in the last step we use the well-known fact \cite[(5.55)]{GKP94} that
\begin{equation}\label{ank2}
\sum_{k=0}^{n}(-1)^k\binom{n}{k}^2=\begin{cases}(-1)^{\frac n2}\binom{n}{\frac n2},&\text{if }n\text{ is even},\\
0,&\text{if }n\text{ is odd}.
\end{cases}
\end{equation}

Suppose that $p\equiv7\pmod{8}$. By the Lucas theorem (cf. \cite[(1)]{Granville97}), we have
$$
\binom{\frac{3p-1}{2}}{k}\equiv\binom{\frac{3p-1}{2}}{p+k}\equiv\binom{\frac{p-1}{2}}{k}\pmod{p}
$$
for $0\leq k\leq\frac{p-1}2$. Hence
$$
\sum_{k=0}^{\frac{3p-1}2}\binom{\frac{3p-1}2}{k}^2z^k\equiv
(1+z^p)\sum_{k=0}^{\frac{p-1}2}\binom{\frac{p-1}2}{k}^2z^k\pmod{p}.
$$
Note that
$$
1+z^p\equiv(1+z)^p=(4+2\sqrt{2})^p\pmod{p}.
$$
Then $1+z^p$ is prime to $p$ since $4+2\sqrt{2}$ is prime to $p$.
Using (\ref{qthf1}) with $a=b=\frac{3p-1}{2}$, we obtain that
\begin{align*}
\sum_{k=0}^{\frac{3p-1}2}\binom{\frac{3p-1}2}{k}^2z^k=&
{}_2F_1\bigg(\begin{matrix}-\frac{3p-1}2&-\frac{3p-1}2\\ &1\end{matrix}\bigg|z\bigg)\\
=&(1-z)^{\frac{3p-1}{2}}{}_2F_1\bigg(\begin{matrix}-\frac{3p-1}4&\frac{3p+1}4\\ &1\end{matrix}\bigg|-1\bigg).
\end{align*}
In view of (\ref{ank2}), we have
\begin{align*}
{}_2F_1\bigg(\begin{matrix}-\frac{3p-1}4&\frac{3p+1}4\\ &1\end{matrix}\bigg|-1\bigg)
=&\sum_{k=0}^{\frac{3p-1}4}\binom{\frac{3p-1}4}{k}\binom{-\frac{3p+1}4}{k}(-1)^k\\
\equiv&\sum_{k=0}^{\frac{3p-1}4}\binom{\frac{3p-1}4}{k}^2(-1)^k=0\pmod{p}.
\end{align*}
So (\ref{alphacong1}) is also valid when $p\equiv 7\pmod{8}$.

(ii) We shall use another quadratic transformation as follows \cite[(3.1.9)]{AAR99}:
\begin{equation}\label{qthf2}
{}_2F_1\bigg(\begin{matrix}a&b\\ &a-b+1\end{matrix}\bigg|z\bigg)=(1+z)^{-a}{}_2F_1\bigg(\begin{matrix}\frac12a&\frac12a+\frac12\\ &a-b+1\end{matrix}\bigg|\frac{4z}{(1+z)^2}\bigg).
\end{equation}
Let $z=-(2+\sqrt{3})^2$. Then we have
$$\frac{4z}{(1+z)^2}=-\frac13.$$
Applying (\ref{qthf2}) with $a=b=-\frac{p-1}2$, we get that
$$\sum_{k=0}^{\frac{p-1}2}\bi{\frac{p-1}2}k^2z^k={}_2F_1\bigg(\begin{matrix}-\frac{p-1}2&-\frac{p-1}2\\ &1\end{matrix}\bigg|z\bigg)=(1+z)^{\frac{p-1}2}{}_2F_1\bigg(\begin{matrix}-\frac{p-1}4&-\frac{p-3}4\\ &1\end{matrix}\bigg|-\frac13\bigg).$$

It suffices to show that
$$
{}_2F_1\bigg(\begin{matrix}-\frac{p-1}4&-\frac{p-3}4\\ &1\end{matrix}\bigg|-\frac13\bigg)\equiv0\pmod{p}
$$
when $p\equiv 2\pmod{3}$.
Note that
\begin{align*}
{}_2F_1\bigg(\begin{matrix}-\frac{p-1}4&-\frac{p-3}4\\ &1\end{matrix}\bigg|-\frac13\bigg)=&\sum_{k=0}^{\frac{p-3}4}\frac{(-\frac{p-1}{4})_k(-\frac{p-3}{4})_k}{(1)_kk!}\cdot\bigg(-\frac13\bigg)^k
\\\equiv&\sum_{k=0}^{\f{p-3}4}\frac{(-\frac{p-1}{4})_k(-\frac{p-3}{4})_k}{\big(\frac p2+1\big)_kk!}\cdot\bigg(-\frac13\bigg)^k\\
=&{}_2F_1\bigg(\begin{matrix}-\frac{p-1}4&-\frac{p-1}4+\frac12\\ &\f p2+1\end{matrix}\bigg|-\frac13\bigg)\pmod p.
\end{align*}
It is known \cite[15.4.31]{NIST}  that
$${}_2F_1\bigg(\begin{matrix}a&a+\frac12\\ &\frac32-2a\end{matrix}\bigg|-\frac13\bigg)=\l(\frac89\r)^{-2a}\cdot\frac{\Gamma{(\frac43)}\Gamma{(\frac32-2a)}}{\Gamma{(\frac32)}\Gamma{(\frac43-2a)}}.$$
Thus
\begin{align*}
{}_2F_1\bigg(\begin{matrix}-\frac{p-1}4&-\frac{p-1}4+\frac12\\ &\frac  p2+1\end{matrix}\bigg|-\frac13\bigg)
=&\bigg(\frac89\bigg)^{\frac{p-1}2}\cdot\frac{\Gamma{(\frac43)}\Gamma{(\frac32+\frac{p-1}2)}}{\Gamma{(\frac32)}\Gamma{(\frac43+\frac{p-1}2)}}.
\end{align*}
When $p\equiv 2\pmod{3}$, $3j+4\neq p$ for any $0\leq j\leq\frac{p-1}2-1$.
But $2j+3=p$ if $j=\frac{p-1}2-1$. So for prime $p\equiv2\pmod{3}$, we always have
$$
{}_2F_1\bigg(\begin{matrix}-\frac{p-1}4&-\frac{p-3}4\\ &\frac  p2+1\end{matrix}\bigg|-\frac13\bigg)=\frac{8^{\frac{p-1}2}}{9^{\frac{p-1}2}}\prod_{j=0}^{\frac{p-1}2-1}\frac{\frac32+j}{\frac43+j}\equiv0\pmod p.$$

(iii)
According to \cite[(3.1.11)]{AAR99}, we have
\begin{equation}\label{qthf3}
{}_2F_1\bigg(\begin{matrix}a&b\\ &a-b+1\end{matrix}\bigg|z^2\bigg)=
(1+z)^{-2a}{}_2F_1\bigg(\begin{matrix}a&a-b+\frac12\\ &2a-2b+1\end{matrix}\bigg|\frac{4z}{(1+z)^2}\bigg).
\end{equation}
Let $z=-(3+2\sqrt{2})$. Apparently
$$
\frac{4z}{(1+z)^2}=-1.
$$
It follows from (\ref{qthf3}) that
\begin{align*}
\sum_{k=0}^{\frac{p-1}2}\binom{\frac{p-1}2}{k}^2z^{2k}=&
{}_2F_1\bigg(\begin{matrix}-\frac{p-1}2&-\frac{p-1}2\\ &1\end{matrix}\bigg|z^2\bigg)\\
=&
(1+z)^{p-1}{}_2F_1\bigg(\begin{matrix}-\frac{p-1}2&\frac12\\ &1\end{matrix}\bigg|-1\bigg).
\end{align*}
If $p\equiv 3\pmod{4}$, then
\begin{align*}
{}_2F_1\bigg(\begin{matrix}-\frac{p-1}2&\frac12\\ &1\end{matrix}\bigg|-1\bigg)=&\sum_{k=0}^{\frac{p-1}{2}}\binom{\frac{p-1}{2}}{k}
\binom{-\frac{1}{2}}{k}(-1)^k\\
\equiv&
\sum_{k=0}^{\frac{p-1}{2}}\binom{\frac{p-1}{2}}{k}^2(-1)^k=0\pmod{p}.
\end{align*}
Thus (\ref{alphacong3}) is also confirmed.
\end{proof}
\begin{lemma} Suppose that $p$ is a prime.
\medskip

{\rm (i)} If $p\equiv 7\pmod{8}$, then
\begin{equation}\label{p2kh78}
\sum_{k=0}^{\frac{p-1}2}\binom{\frac{p-1}2}{k}^2P_{2k}\equiv0\pmod{p^2}.
\end{equation}

{\rm (ii)} If $p\equiv 11\pmod{12}$, then
\begin{equation}\label{r2kh1112}
\sum_{k=0}^{\frac{p-1}2}\binom{\frac{p-1}2}{k}^2(-1)^kR_{2k}\equiv0\pmod{p^2}.
\end{equation}

{\rm (iii)} If $p\equiv 7\pmod{8}$, then
\begin{equation}\label{w2kh78}
\sum_{k=0}^{\frac{p-1}2}\binom{\frac{p-1}2}{k}^2\frac{W_{2k}}{2^k}\equiv0\pmod{p^2}.
\end{equation}
\end{lemma}
\begin{proof} (i)
Let $\alpha=1+\sqrt{2}$ and $\beta=1-\sqrt{2}$. We know that
$$
P_k=\frac{\alpha^k-\beta^k}{\alpha-\beta}.
$$
Clearly
$$
\sum_{k=0}^{\frac{p-1}2}\binom{\frac{p-1}2}{k}^2\beta^{2k}=
\sum_{k=0}^{\frac{p-1}2}\binom{\frac{p-1}2}{k}^2\beta^{p-1-2k}=
\beta^{p-1}\sum_{k=0}^{\frac{p-1}2}\binom{\frac{p-1}2}{k}^2\alpha^{2k}.
$$
Hence
$$
\sum_{k=0}^{\frac{p-1}2}\binom{\frac{p-1}2}{k}^2(\alpha^{2k}-\beta^{2k})=
(1-\beta^{p-1})\sum_{k=0}^{\frac{p-1}2}\binom{\frac{p-1}2}{k}^2\alpha^{2k}.
$$
If $p\equiv7\pmod{8}$, then
$$
2^{\frac{p-1}{2}}\equiv1\pmod{p}.
$$
It follows that
\begin{equation}\label{betap1}
\beta^{p-1}=\frac{(1-\sqrt{2})^p}{1-\sqrt{2}}\equiv
\frac{1-2^{\frac{p-1}2}\cdot\sqrt{2}}{1-\sqrt{2}}\equiv\frac{1-\sqrt{2}}{1-\sqrt{2}}\equiv1\pmod{p}.
\end{equation}
Similarly, we also have $\alpha^{p-1}\equiv1\pmod{p}$.
In view of (\ref{alphacong1}) and (\ref{betap1}), when $p\equiv7\pmod{8}$, we can get
\begin{align*}
\sum_{k=0}^{\frac{p-1}2}\binom{\frac{p-1}2}{k}^2P_{2k}=&
\frac1{\alpha-\beta}\sum_{k=0}^{\frac{p-1}2}\binom{\frac{p-1}2}{k}^2(\alpha^{2k}-\beta^{2k})\\
=&\frac{1-\beta^{p-1}}{\alpha-\beta}\sum_{k=0}^{\frac{p-1}2}\binom{\frac{p-1}2}{k}^2\alpha^{2k}
\equiv0\pmod{p^2}.
\end{align*}

(ii) Let $\alpha=2+\sqrt{3}$ and $\beta=2-\sqrt{3}$. Then
$$
R_k=\alpha^k+\beta^k.
$$
Suppose that $p\equiv11\pmod{12}$. By the quadratic reciprocity theorem,
$$
\jacob{3}p=(-1)^{\frac{p-1}{2}\cdot\frac{3-1}{2}}\jacob{p}{3}=(-1)\cdot(-1)=1.
$$
So $3^{\frac{p-1}2}\equiv1\pmod{p}$.
Similarly as (\ref{betap1}), we can get
$$
\alpha^{p-1}\equiv\beta^{p-1}\equiv1\pmod{p}.
$$
Furthermore,
\begin{align*}
\sum_{k=0}^{\frac{p-1}2}\binom{\frac{p-1}2}{k}^2(-1)^k\beta^{2k}=&
\sum_{k=0}^{\frac{p-1}2}\binom{\frac{p-1}2}{k}^2(-1)^{\frac{p-1}{2}-k}\beta^{p-1-2k}\\
=&-\beta^{p-1}\sum_{k=0}^{\frac{p-1}2}\binom{\frac{p-1}2}{k}^2(-1)^k\alpha^{2k}.
\end{align*}
It follows from (\ref{alphacong2}) that
\begin{align*}
\sum_{k=0}^{\frac{p-1}2}\binom{\frac{p-1}2}{k}^2(-1)^k(\alpha^{2k}+\beta^{2k})=&
(1-\beta^{p-1})\sum_{k=0}^{\frac{p-1}2}\binom{\frac{p-1}2}{k}^2(-1)^k\alpha^{2k}\\
\equiv&0\pmod{p^2}.
\end{align*}

(iii) Let $\alpha=3+2\sqrt{2}$ and $\beta=3-2\sqrt{2}$. It is easy to verify that
\begin{equation}\label{walphabeta}
W_{2k}=2^{k+2}\cdot\frac{\alpha^{2k}-\beta^{2k}}{\alpha^2-\beta^2}
\end{equation}
for each $k\geq 0$.
If $p\equiv 7\pmod{8}$, we have $\alpha^{p-1}\equiv\beta^{p-1}\equiv 1\pmod{p}$, too.
Similarly, we can get
\begin{align*}
\sum_{k=0}^{\frac{p-1}2}\binom{\frac{p-1}2}{k}^2\frac{W_{2k}}{2^k}=&\frac{4}{\alpha^2-\beta^2}\sum_{k=0}^{\frac{p-1}2}\binom{\frac{p-1}2}{k}^2(\alpha^{2k}-\beta^{2k})\\
=&\frac{1-\beta^{p-1}}{6\sqrt{2}}\sum_{k=0}^{\frac{p-1}2}\binom{\frac{p-1}2}{k}^2\alpha^{2k}\equiv0\pmod{p^2}.
\end{align*}

\end{proof}

\section{Proofs of Theorems \ref{thm1} and \ref{thm1n}}
\setcounter{lemma}{0}
\setcounter{theorem}{0}
\setcounter{corollary}{0}
\setcounter{remark}{0}
\setcounter{equation}{0}
\setcounter{conjecture}{0}

We firstly consider (\ref{cong2}). It is easy to check that $V_n(-4,4)=-(-2)^{n+1}$ and $V_n(-2,1)=2(-1)^{n}$. Substituting $a=-4$ and $b=4$ in Theorem \ref{thm2}, we get
\begin{align*}
&\sum_{k=0}^{\frac{p-1}{2}}\binom{p-1}{k}\binom{2k}{k}^2\frac{1}{(-8)^k}\\
\equiv&(-1)^{\frac{p-1}{2}}16^{p-1}\sum_{j=0}^{\frac{p-1}{2}}(-1)^j\binom{\frac{p-1}{2}}{j}^2(1+2pH_{2j})
\pmod{p^2}.
\end{align*}
Suppose that $p\equiv 3\pmod{4}$. By (\ref{ank2}),
$$
\sum_{j=0}^{\frac{p-1}{2}}(-1)^j\binom{\frac{p-1}{2}}{j}^2=0.
$$
Note that for any $1\leq j\leq p-2$,
\begin{equation}\label{hp1j}
H_{p-1-j}=\sum_{i=j+1}^{p-1}\frac1{p-i}\equiv -(H_{p-1}-H_j)\equiv H_j\pmod{p}.
\end{equation}
We have
\begin{align*}
\sum_{j=0}^{\frac{p-1}{2}}(-1)^j\binom{\frac{p-1}{2}}{j}^2H_{2j}=
&(-1)^{\frac{p-1}2}\sum_{j=0}^{\frac{p-1}{2}}(-1)^{j}\binom{\frac{p-1}{2}}{j}^2H_{p-1-2j}\\
\equiv&
-\sum_{j=0}^{\frac{p-1}{2}}(-1)^{j}\binom{\frac{p-1}{2}}{j}^2H_{2j}\pmod{p},
\end{align*}
which clearly implies
$$
\sum_{j=0}^{\frac{p-1}{2}}(-1)^j\binom{\frac{p-1}{2}}{j}^2H_{2j}\equiv 0\pmod{p}.
$$
Therefore
$$
\sum_{k=0}^{\frac{p-1}{2}}\binom{p-1}{k}\binom{2k}{k}^2\frac{1}{(-8)^k}\equiv0\pmod{p^2}$$
for any prime $p\equiv3\pmod{4}$.

Let us consider (\ref{cong3}). Evidently $U_n(-1,1)=\chi_3(n)$ and $U_n(1,1)=(-1)^{n-1}\chi_3(n)$.
Applying Theorem \ref{thm2} with $a=-1$ and $b=1$, we obtain that
\begin{align*}
&\sum_{k=0}^{\frac{p-1}{2}}\binom{p-1}{k}\binom{2k}{k}^2\frac{\chi_3(k)}{16^k}\\
\equiv&(-1)^{\frac{p-1}{2}}16^{p-1}\sum_{j=0}^{\frac{p-1}{2}}(-1)^{j-1}\chi_3(j)\binom{\frac{p-1}{2}}{j}^2(1+2pH_{2j})
\pmod{p^2}.
\end{align*}
Suppose that the prime $p\equiv 1\pmod{12}$. Then
$$
(-1)^{\frac{p-1}{2}}\chi_3\bigg(\frac{p-1}{2}-j\bigg)=
\chi_3(-j)=-\chi_3(j).
$$
It follows that
$$
\sum_{j=0}^{\frac{p-1}{2}}(-1)^j\chi_3(j)\binom{\frac{p-1}{2}}{j}^2=
-\sum_{j=0}^{\frac{p-1}{2}}(-1)^j\chi_3(j)\binom{\frac{p-1}{2}}{j}^2,
$$
i.e.,
$$
\sum_{j=0}^{\frac{p-1}{2}}(-1)^j\chi_3(j)\binom{\frac{p-1}{2}}{j}^2=0.
$$
Similarly, we also have
$$
\sum_{j=0}^{\frac{p-1}{2}}(-1)^j\chi_3(j)\binom{\frac{p-1}{2}}{j}^2H_{2j}\equiv 0\pmod{p}.
$$
Thus (\ref{cong3}) is also proved.

The proofs of (\ref{congm}) and (\ref{congd3}) are very similar as the one of (\ref{cong3}).
Clearly $(-1)^{k-1}M_k=U_k(-3,3)$. Suppose that $p\equiv1\pmod{6}$. By Theorem \ref{thm2}, we only need to show that
\begin{align*}
\sum_{j=0}^{\frac{p-1}{2}}\binom{\frac{p-1}{2}}{j}^2U_k(-1,1)(1+2pH_{2j})\equiv0\pmod{p^2}.
\end{align*}
Since $U_k(-1,1)=\chi_3(k)$ and $\chi_3(\frac{p-1}2-k)=-\chi_3(k)$ now, we can get
$$
\sum_{j=0}^{\frac{p-1}{2}}\chi_3(j)\binom{\frac{p-1}{2}}{j}^2=0
\qquad\text{and}\qquad
\sum_{j=0}^{\frac{p-1}{2}}\chi_3(j)\binom{\frac{p-1}{2}}{j}^2H_{2j}\equiv0\pmod{p}.
$$
Then (\ref{congm}) is derived. Let
$$
\delta_3(k)=\begin{cases}2,&\text{if }k\equiv 0\pmod{3},\\
-1,&\text{otherwise}.
\end{cases}
$$
Obviously (\ref{congd3}) is equivalent to
$$
\sum_{j=0}^{\frac{p-1}{2}}\binom{p-1}{j}\binom{2j}{j}^2\frac{\delta_3(j)}{16^j}\equiv0\pmod{p^2}
$$
for each prime $p\equiv7\pmod{12}$.
Since $\delta_3(k)=V_k(-1,1)$, it suffices to show that
$$
\sum_{j=0}^{\frac{p-1}{2}}\binom{\frac{p-1}{2}}{j}^2V_j(1,1)(1+2pH_{2j})\equiv0\pmod{p^2}.
$$
It is easy to verify $$V_k(1,1)=-V_{6h+3-k}(1,1)$$ for any $h\geq 0$ and $0\leq k\leq 6h+3$.
So if $\frac{p-1}{2}\equiv 3\pmod{6}$,
$$
\sum_{j=0}^{\frac{p-1}{2}}\binom{\frac{p-1}{2}}{j}^2V_j(1,1)=0
\quad\text{and}\quad
\sum_{j=0}^{\frac{p-1}{2}}\binom{\frac{p-1}{2}}{j}^2V_j(1,1)H_{2j}\equiv0\pmod{p}.$$

Finally, let us turn to (\ref{congp}), (\ref{congr}) and (\ref{congw}). We require some additional auxiliary results.
\begin{lemma} (i) Suppose that $p\equiv\pm1\pmod{8}$ is a prime.
Then
\begin{equation}\label{ph2k0}
\sum_{k=0}^{\frac{p-1}2}\binom{\frac{p-1}2}{k}^2P_{2k}H_{2k}\equiv0\pmod{p}.
\end{equation}

\noindent(ii) Suppose that $p\equiv11\pmod{12}$ is a prime.
Then
\begin{equation}\label{rh2k0}
\sum_{k=0}^{\frac{p-1}2}\binom{\frac{p-1}2}{k}^2(-1)^kR_{2k}H_{2k}\equiv0\pmod{p}.
\end{equation}

\noindent(iii) Suppose that $p\equiv7\pmod{8}$ is a prime.
Then
\begin{equation}\label{wh2k0}
\sum_{k=0}^{\frac{p-1}2}\binom{\frac{p-1}2}{k}^2\frac{W_{2k}}{2^k}\cdot H_{2k}\equiv0\pmod{p}.
\end{equation}
\end{lemma}
\begin{proof} (i)
Let $\alpha=1+\sqrt{2}$ and $\beta=1-\sqrt{2}$.
Clearly
\begin{align*}
\sum_{k=0}^{\frac{p-1}2}\binom{\frac{p-1}2}{k}^2(\alpha^{2k}-\beta^{2k})H_{2k}
=&\sum_{k=0}^{\frac{p-1}2}\binom{\frac{p-1}2}{k}^2(\alpha^{p-1-2k}-\beta^{p-1-2k})H_{p-1-2k}\\
=&\sum_{k=0}^{\frac{p-1}2}\binom{\frac{p-1}2}{k}^2(\alpha^{p-1}\beta^{2k}-\beta^{p-1}\alpha^{2k})H_{p-1-2k}.
\end{align*}
Recall that $\alpha^{p-1}\equiv\beta^{p-1}\equiv1\pmod{p}$ now.
In view of (\ref{hp1j}), we get
$$
\sum_{k=0}^{\frac{p-1}2}\binom{\frac{p-1}2}{k}^2(\alpha^{2k}-\beta^{2k})H_{2k}
\equiv\sum_{k=0}^{\frac{p-1}2}\binom{\frac{p-1}2}{k}^2(\beta^{2k}-\alpha^{2k})H_{2k}\pmod{p},
$$
i.e.,
$$
\sum_{k=0}^{\frac{p-1}2}\binom{\frac{p-1}2}{k}^2(\alpha^{2k}-\beta^{2k})H_{2k}
\equiv0\pmod{p}.
$$

\noindent(ii) Let $\alpha=2+\sqrt{3}$ and $\beta=2-\sqrt{3}$. Similarly, we have
\begin{align*}
&\sum_{k=0}^{\frac{p-1}2}\binom{\frac{p-1}2}{k}^2(-1)^k(\alpha^{2k}+\beta^{2k})H_{2k}\\
=&\sum_{k=0}^{\frac{p-1}2}\binom{\frac{p-1}2}{k}^2(-1)^{\frac{p-1}{2}-k}(\alpha^{p-1}\beta^{2k}+\beta^{p-1}\alpha^{2k})H_{p-1-2k}\\
\equiv&
-\sum_{k=0}^{\frac{p-1}2}\binom{\frac{p-1}2}{k}^2(-1)^{k}(\beta^{2k}+\alpha^{2k})H_{2k}\pmod{p},
\end{align*}
which clearly implies (\ref{rh2k0}).

\noindent(iii) Let $\alpha=3+2\sqrt{3}$ and $\beta=3-2\sqrt{3}$. Since $p\equiv\pm1\pmod{8}$,
we get
$$
\sum_{k=0}^{\frac{p-1}2}\binom{\frac{p-1}2}{k}^2(\alpha^{2k}-\beta^{2k})H_{2k}
\equiv0\pmod{p}.
$$
Then (\ref{wh2k0}) is immediately derived from (\ref{walphabeta}).
\end{proof}
Now we are ready to prove (\ref{congp}), (\ref{congr}) and (\ref{congw}).
It is not difficult to see that $2^nP_n=U_n(4,-4)$ and $P_{2n}=2U_n(6,1)$. So by (\ref{congun}), (\ref{p2kh78}) and (\ref{ph2k0}),
\begin{align*}
&\sum_{k=0}^{\frac{p-1}{2}}\binom{p-1}{k}\binom{2k}{k}^2\frac{P_k}{8^k}\\
\equiv&(-1)^{\frac{p-1}{2}}16^{p-1}\bigg(\sum_{j=0}^{\frac{p-1}{2}}\binom{\frac{p-1}{2}}{j}^2P_{2j}+2p\sum_{j=0}^{\frac{p-1}{2}}\binom{\frac{p-1}{2}}{j}^2P_{2j}H_{2j}\bigg)\equiv0
\pmod{p^2}.
\end{align*}
(\ref{congr}) similarly follows from (\ref{r2kh1112}) and (\ref{rh2k0}), since $(-4)^nR_n=V_n(-16,16)$ and $(-1)^nR_{2n}=V_n(-14,1)$. Easily we can verify $4^{n-1}W_n=U_n(32,-32)$ and $2^{-n-2}W_{2n}=U_n(34,1)$. So (\ref{congw}) is also an easy consequence of (\ref{w2kh78}) and (\ref{wh2k0}).

\begin{acknowledgment}
We are grateful to the anonymous referee for his/her very useful suggestions. We also thank Professor Zhi-Wei Sun for his helpful comments on this paper.
\end{acknowledgment}


\begin{thebibliography}{999}

\bibitem{AAR99} G. E. Andrews, R. Askey and R. Roy, {\it Special Functions},
Cambridge University Press, Cambridge, 1999.


\bibitem{GKP94} R. L. Graham, D. E. Knuth and O. Patashnik, {\it Concrete mathematics: A foundation for computer science}, Second edition, Addison-Wesley Publishing Company, Reading, MA, 1994.

\bibitem{Granville97} A. Granville, {\it Arithmetic properties of binomial coefficients. I: Binomial coefficients modulo prime powers}, Organic mathematics (Burnaby, BC, 1995), 253--276, CMS Conf. Proc., 20, Amer. Math. Soc., Providence, RI, 1997.

\bibitem{GZ14} Victor J. W. Guo and J. Zeng, {\it Some $q$-analogues of supercongruences of Rodriguez-Villegas} J. Number Theory,  {\bf 145} (2014), 301-316.

\bibitem{Lehmer38} E. Lehmer, {\it On congruences involving Bernoulli numbers and the quotients of Fermat and Wilson},
Annals of Math., {\bf 39}(1938), 350-360.

\bibitem{Long11} L. Long, {\it Hypergeometric evaluation identities and supercongruences}, Pacific J. Math., {\bf 249} (2011),  405-418.

\bibitem{MT13} S. Mattarei and R. Tauraso, {\it Congruences for central binomial sums and finite polylogarithms}, J. Number Theory, {\bf 133} (2013), 131-157.

\bibitem{McCarthy12} D. McCarthy, {\it On a supercongruence conjecture of Rodriguez-Villegas}, Proc. Amer. Math. Soc., {\bf 140} (2012), 2241-2254.

\bibitem{Morley95} F. Morley, {\it Note on the congruence $2^{4n}\equiv(-1)^n(2n)!/(n!)^2$, where $2n+1$ is a prime},
Annals of Math., {\bf 9}(1895), 168-170.

\bibitem{Mortenson03a} E. Mortenson, {\it  A supercongruence conjecture of Rodriguez-Villegas for a certain truncated hypergeometric function}, J. Number Theory, {\bf 99} (2003), 139-147.

\bibitem{Mortenson05}    E. Mortenson,  {\it Supercongruences for truncated ${}_{n+1}F_n$ hypergeometric series with applications to certain weight three newforms}, Proc. Amer. Math. Soc., {\bf 133} (2005), 321-330.

\bibitem{NIST} {\it NIST handbook of mathematical functions}, Edited by Frank W. J. Olver, Daniel W. Lozier, Ronald F. Boisvert and Charles W. Clark, U.S. Department of Commerce, National Institute of Standards and Technology, Washington, DC; Cambridge University Press, Cambridge, 2010.

\bibitem{OS09} R. Osburn and C. Schneider, {\it Gaussian hypergeometric series and supercongruences}, Math. Comp., {\bf 78} (2009), 275-292.

\bibitem{Ribenboim00} P. Ribenboim, {\it My numbers, my friends: Popular lectures on number theory}, Springer-Verlag, New York, 2000.

\bibitem{RV03} F. Rodriguez-Villegas, {\it Hypergeometric families of Calabi-Yau manifolds. Calabi-Yau Varieties and
Mirror Symmetry (Yui, Noriko (ed.) et al., Toronto, ON, 2001)}, 223-231, Fields Inst. Commun., {\bf 38}, Amer. Math. Soc., Providence, RI, (2003).

\bibitem{Sh11}      Z.-H. Sun, {\it Congruences concerning Legendre polynomials}, Proc. Amer. Math. Soc., {\bf 139} (2011),  1915--1929.

\bibitem{Sh13} Z.-H. Sun, {\it Congruences concerning Legendre polynomials II}, J. Number Theory, {\bf 133} (2013), 1950-1976.

\bibitem{Sh14}    Z.-H. Sun, {\it Generalized Legendre polynomials and related supercongruences}, J. Number Theory, {\bf 143} (2014), 293-319.

\bibitem{Sw09}  Z.-W. Sun, {\it Open conjectures on congruences}, preprint, {\tt arxiv:0911.5665}.

\bibitem{Sw11a}  Z.-W. Sun, {\it On congruences related to central binomial coefficients}, J. Number Theory, {\bf 131} (2011), 2219-2238.

\bibitem{Sw11b}  Z.-W. Sun, {\it Super congruences and Euler numbers}, Sci. China Math. {\bf 54} (2011), 2509-2535.

\end{thebibliography}
\end{document}